\documentclass[12pt]{amsart}
\usepackage{amsmath,amsthm,amsfonts,amssymb,mathrsfs}
\date{\today}
\usepackage{color}

\newtheorem{theorem}{Theorem}[section]
\newtheorem{proposition}[theorem]{Proposition}
\newtheorem{corollary}[theorem]{Corollary}
\newtheorem{lemma}[theorem]{Lemma}

\theoremstyle{definition}
\newtheorem{problem}[theorem]{Problem}

\newcommand\w{\omega}

\newcommand\C{\mathcal{C}}
\usepackage{hyperref}

\begin{document}

\title[On the lattice of weak topologies on the bicyclic monoid with adjoined zero]{On the lattice of weak topologies on the bicyclic monoid with adjoined zero}

\author[S.~Bardyla]{Serhii~Bardyla}
\thanks{The work of the author is supported by the Austrian Science Fund FWF (Grant  I
3709 N35).}
\address{S. Bardyla: Institute of Mathematics, Kurt G\"{o}del Research Center, Vienna, Austria}
\email{sbardyla@yahoo.com}

\author[O~Gutik]{Oleg~Gutik}
\address{Department of Mechanics and Mathematics,
National University of Lviv, Universytetska 1, Lviv, 79000, Ukraine}
\email{oleg.gutik@lnu.edu.ua, {ovgutik}@yahoo.com}



\begin{abstract}
A Hausdorff topology $\tau$ on the bicyclic monoid with adjoined zero $\mathcal{C}^0$ is called {\em weak} if it is contained in the coarsest inverse semigroup topology on $\mathcal{C}^0$.
We show that the lattice $\mathcal{W}$ of all weak shift-continuous topologies on $\mathcal{C}^0$ is isomorphic to the lattice $\mathcal{SIF}^1{\times}\mathcal{SIF}^1$ where $\mathcal{SIF}^1$ is a set of all shift-invariant filters on $\omega$ with an attached element $1$ endowed with the following partial order: $\mathcal{F}\leq \mathcal{G}$ iff $\mathcal{G}=1$ or $\mathcal{F}\subset \mathcal{G}$. Also, we investigate  cardinal characteristics of the lattice $\mathcal{W}$. In particular, we proved that $\mathcal{W}$ contains an antichain of cardinality $2^{\mathfrak{c}}$ and a well-ordered chain of cardinality $\mathfrak{c}$. Moreover, there exists a well-ordered chain of first-countable weak topologies of order type $\mathfrak{t}$.
\end{abstract}

\maketitle

\section{Introduction and preliminaries}
In this paper all topological spaces are assumed to be Hausdorff. The cardinality of a set $X$ is denoted by $|X|$.
Further we shall follow the terminology of \cite{Engelking-1989, Kunen, Lawson-2009}. By $\omega$ we denote the first infinite ordinal. The set of integers is denoted by $\mathbb{Z}$. By $\mathfrak{c}$ we denote the cardinality of the family of all subsets of $\omega$.

A semigroup $S$ is called \emph{inverse} if for every $x\in S$ there exists a unique $y\in S$ such that $xyx=x$ and $yxy=y$. Such an element $y$ is denoted by $x^{-1}$ and called the \emph{inverse of} $x$. The map which associates every element of an inverse semigroup to its
inverse is called an \emph{inversion}.

Given a semigroup $S$, we shall
denote the set of all idempotents of $S$ by $E(S)$.
A semigroup $S$ with an adjoined zero will be denoted by
$S^0$.

A {\em poset} is a set endowed with a partial order which is usually denoted by $\leq$. A poset $X$ is called {\em upper} ({\em lower}, resp.) {\em semilattice} if each finite subset of $X$ has a supremum (infimum, resp.). An upper (lower, resp.) semilattice $X$ is called {\em complete} if each subset of $X$ has a supremum (infimum, resp.). Elements $x,y$ of a poset $X$ are called {\em incomparable} if neither $x\leq y$ nor $y\leq x$.
For any poset $X$ and $x\in X$ put
$${\downarrow}x=\{y\in X\mid y\leq x\}, \qquad {\uparrow}x=\{y\in X\mid x\leq y\},\qquad{\downarrow}^\circ x={\downarrow}x\setminus \{x\},
\qquad{\uparrow}^\circ x={\uparrow}x\setminus \{x\}.$$
A poset $X$ is called a {\em lattice} if each finite non-empty subset of $X$ has infimum and supremum. A lattice $X$ is called {\em complete} if any non-empty subset of $X$ has infimum and supremum. A poset $X$ is called a {\em chain} if for any distinct $a,b\in X$ either $a\leq b$ or $b\leq a$. A chain $X$ is called {\em well-ordered} if any non-empty subset $A$ of $X$ contains the smallest element.

If $Y$ is a subset of a topological space $X$, then by $\overline{Y}$ we denote the closure of $Y$ in $X$.

A family $\mathcal{F}$ of subsets of a set $X$ is called a {\em filter} if it satisfies the following conditions:
\begin{itemize}
\item[$(1)$] $\emptyset\notin \mathcal{F}$;
\item[$(2)$] If $A\in \mathcal{F}$ and $A\subset B$ then $B\in \mathcal{F}$;
\item[$(3)$] If $A,B\in \mathcal{F}$ then $A\cap B\in\mathcal{F}$.
\end{itemize}
A family $\mathcal{B}$ is called a {\em base} of a filter $\mathcal{F}$ if for each element $A\in\mathcal{F}$ there exists an element $B\in\mathcal{B}$ such that $B\subset A$. A filter $\mathcal{F}$ is called {\em free} if $\cap_{F\in\mathcal{F}}=\emptyset$.

A \emph{semitopological} (\emph{topological}, resp.) \emph{semigroup} is a topological space together with a separately (jointly, resp.) continuous  semigroup operation. An inverse semigroup with continuous semigroup operation and inversion is called a \emph{topological inverse semigroup}.

A topology $\tau$ on a semigroup $S$ is called:
\begin{itemize}
  \item \emph{shift-continuous}, if $(S,\tau)$ is a semitopological semigroup;
  \item \emph{semigroup}, if $(S,\tau)$ is a topological semigroup;
  \item \emph{inverse semigroup}, if $(S,\tau)$ is a topological inverse semigroup.
\end{itemize}

The \emph{bicyclic monoid} $\mathcal{C}$ is a semigroup with the identity $1$ generated by two elements $p$ and $q$ subject to the condition $pq=1$.

The bicyclic monoid is isomorphic to the set $\w{\times}\w$ endowed with the following semigroup operation:
\begin{equation*}
    (a,b)\cdot (c,d)=
    \left\{
      \begin{array}{cl}
        (a+c-b,d), & \hbox{if~~} b\leq c;\\
        (a,d+b-c),   & \hbox{if~~} b>c;
      \end{array}
    \right.
\end{equation*}
The bicyclic monoid 
plays an important role in the algebraic theory of semigroups as well as in the
theory of topological semigroups. For example, the well-known
Andersen's result~\cite{Andersen-1952} states that a ($0$--)simple
semigroup with an idempotent is completely ($0$--)simple if and only if it does not
contain an isomorphic copy of the bicyclic semigroup. The bicyclic semigroup admits only the
discrete semigroup topology~\cite{Eberhart-Selden-1969}. In~\cite{Bertman-West-1976} this result was extended over the case of semitopological semigroups.
Compact topological semigroups cannot contain an isomorphic copy of the bicyclic monoid~\cite{Anderson-Hunter-Koch-1965}.
The problem of an embedding of the bicyclic monoid into compact-like topological semigroups was discussed in~\cite{Banakh-Dimitrova-Gutik-2009, Banakh-Dimitrova-Gutik-2010, BR, Gutik-Repovs-2007, Hildebrant-Koch-1988}.

However, it is natural to consider the bicyclic monoid with an adjoint zero $\mathcal{C}^0$.  It is well-known that the bicyclic monoid with an adjoined zero is isomorphic to the polycyclic monoid $\mathcal{P}_1$ which is isomorphic to the graph inverse semigroup $G(E)$ over the graph $E$ which consists of one vertex and one loop. The monoid $\mathcal{C}^0$ is a building block of $\alpha$-bicyclic monoids, polycyclic monoids and some graph inverse semigroups. For example:
\begin{theorem}[{\cite[Theorem 6]{Bardyla-2018(2)}}]\label{main}
Let a graph inverse semigroup $G(E)$ be a dense subsemigroup of a d-compact topological semigroup $S$. Then the following statements hold:
\begin{itemize}
\item[$(1)$] there exists a cardinal $k$ such that $E=(\sqcup_{\alpha\in k} E_{\alpha})\sqcup F$ where the graph
$F$ is acyclic and
for each $\alpha\in k$ the graph $E_{\alpha}$ consists of one vertex and one loop;
\item[$(2)$] if the graph $F$ is non-empty, then the semigroup $G(F)$ is a compact subset of $G(E)$;
\item[$(3)$] each open neighborhood of $0$ contains all but finitely many subsets $G(E_{\alpha})\subset G(E)$, $\alpha\in k$.
\end{itemize}
\end{theorem}

In~\cite{Bardyla-2017(2)} it was proved the following:

\begin{theorem}[{\cite[Theorem 1]{Bardyla-2017(2)}}]
Each graph inverse semigroup $G(E)$ is a subsemigroup of the polycyclic monoid $\mathcal{P}_{|G(E)|}$.
\end{theorem}

This result leads us to the following problem:

\begin{problem}[{\cite[Question 1]{Bardyla-2017(2)}}]
Is it true that each semitopological (topological, topological inverse, resp.) graph inverse semigroup $G(E)$ is a subsemigroup of a semitopological (topological, topological inverse, resp.) polycyclic monoid $\mathcal{P}_{|G(E)|}$.
\end{problem}

So, to solve this problem we need to know more about a topologization of polycyclic monoids and their subsemigroups.
In recent years it was actively investigated locally compact shift-continuous and semigroup topologies on polycyclic monoids and graph inverse semigroups. For instance, in~\cite{Gutik-2015} it was proved that a Hausdorff locally compact semitopological bicyclic semigroup with an adjoined zero $\C^0$ is either compact or discrete. In~\cite{Bardyla-2016} and \cite{Bardyla-2018(1)} this result was extended for polycyclic monoids and graph inverse semigroups over strongly connected graphs with finitely many vertices, respectively. Similar dichotomy also holds for other generalizations of the bicyclic monoid (see~\cite{Bardyla-2018,Gutik-2018,Mokr}).

There exists the coarsest inverse semigroup topology $\tau_{\min}$ on $\C^0$ which is defined as follows: elements of $\C$ are isolated and
the family $\mathcal{B}_{\min}(0)=\{C_n\colon n\in\omega\}$, where $C_n=\{0\}\cup\{(k,m)\mid k,m>n)\}$, forms an open neighborhood base at $0$ of the topology $\tau_{\min}$ (see~\cite[Theorem 3.6]{Bardyla-2017(2)}).

Let $X$ be any semigroup. By $\mathcal{SCT}(X)$ we denote the set of all Hausdorff shift-continuous topologies on $X$ endowed with the following natural partial order: $\tau_1\leq \tau_2$ iff $\tau_1\subset \tau_2$.
For any topologies $\tau_1$, $\tau_2$ on $X$ by $\tau_1{\vee}\tau_2$ we denote a topology on $X$ which subbase is $\tau_1\cup\tau_2$.

\begin{lemma}\label{l0}
For any semigroup $X$, the poset $\mathcal{SCT}(X)$ is a complete upper semilattices. Moreover, if the poset $\mathcal{SCT}(X)$ contains the least element then $\mathcal{SCT}(X)$ is a complete lattice.
\end{lemma}

\begin{proof}
Let $T=\{\tau_{\alpha}\}_{\alpha\in A}$ be an arbitrary subset of $\mathcal{SCT}(X)$. Obviously, the topology $\tau$ which is generated by the subbase $\mathcal{B}_{\tau}=\cup T$ is the supremum of $T$ in the lattice $\mathcal{T}$ of all topologies on $X$. Since $\mathcal{SCT}(X)$ is a subposet of $\mathcal{T}$, it is sufficient to show that $\tau\in \mathcal{SCT}(X)$. Fix an arbitrary $x,y\in X$ and a basic open neighborhood $U$ of $xy$ in the topology $\tau$. Then there exist a finite subset $\{\alpha_1,\ldots, \alpha_n\}\subset A$ and open neighborhoods $U_{\alpha_1}\in \tau_{\alpha_1},\ldots, U_{\alpha_n}\in \tau_{\alpha_n}$ of $xy$ such that $\cap_{i=1}^nU_{\alpha_i}\subset U$. Since for every $i\leq n$ the topology $\tau_{\alpha_{i}}$ is shift-continuous there exists an open neighborhood $V_{\alpha_i}\in \tau_{\alpha_i}$ of $y$ such that $x\cdot V_{\alpha_i}\subset U_{\alpha_i}$. Then $V=\cap_{i=1}^nV_{\alpha_i}\in \tau$ and $x\cdot V\subset U$. Analogously it can be shown that there exists an open neighborhood $W\in \tau$ of $x$ such that $W\cdot y\subset U$. Hence $\tau\in \mathcal{SCT}(X)$.

Assume that the poset $\mathcal{SCT}(X)$ contains the least element $\upsilon$ and $T=\{\tau_{\alpha}\}_{\alpha\in A}$ is an arbitrary subset of $\mathcal{SCT}(X)$.
The set $S=\cap_{\alpha\in A}{\downarrow}\tau_{\alpha}$ is non-empty, because it contains $\upsilon$. The above arguments imply that there exists a topology $\tau$ such that $\tau=\sup S$. The topology $\tau$ is generated by the subbase $\cup S$. Fix any $\alpha\in A$ and a basic open set $U=\cap_{i=1}^nV_i\in \tau$ where $V_i\in \tau_i\in S$, $i\leq n$. Then $V_i\in \tau_{\alpha}$ for each $i\leq n$ witnessing that $U\in \tau_{\alpha}$. Hence $\tau\subset \tau_{\alpha}$ for each $\alpha\in A$ which implies that $\tau=\inf T$.
\end{proof}

By $\mathcal{SCT}$ we denote the poset $\mathcal{SCT}(\mathcal{C}^0)$.

Let $\tau_{c}$ be a topology on $\mathcal{C}^0$ such that each non-zero point is isolated in $(\mathcal{C}^0,\tau_c)$ and an open neighborhood base at $0$ of the topology $\tau_c$ consists of cofinite subsets of $\mathcal{C}^0$ which contain $0$.
We remark that $\tau_c$ is the unique compact shift-continuous topology on $\C^0$ which implies that $\tau_c=\inf\mathcal{SCT}$
(see~\cite[Theorem 1]{Gutik-2015} and~\cite[Lemma 3]{Bardyla-2018(1)} for a more general case).
Hence Lemma~\ref{l0} implies the following:

\begin{corollary}
The poset $\mathcal{SCT}$ is a complete lattice.
\end{corollary}

A topology $\tau$ on $\mathcal{C}^0$ is called {\em weak} if it is contained in the coarsest inverse semigroup topology $\tau_{\min}$ on $\mathcal{C}^0$. By $\mathcal{W}$ we denote the sublattice ${\downarrow} \tau_{\min}\subset \mathcal{SCT}$ of all weak shift-continuous topologies on $\mathcal{C}^0$.

In this paper we investigate properties of the lattice $\mathcal{W}$. More precisely, we show that $\mathcal{W}$ is isomorphic to the lattice $\mathcal{SIF}^1{\times}\mathcal{SIF}^1$ where $\mathcal{SIF}^1$ is a set of all shift-invariant filters on $\omega$ with an attached element $1$ endowed with the following partial order: $\mathcal{F}\leq \mathcal{G}$ iff $\mathcal{G}=1$ or $\mathcal{F}\subset \mathcal{G}$. Also, we investigate cardinal characteristics of the lattice $\mathcal{W}$. In particular, we proved that $\mathcal{W}$ contains an antichain of cardinality $2^{\mathfrak{c}}$ and a well-ordered chain of cardinality $\mathfrak{c}$. Moreover, there exists a well-ordered chain of first-countable weak topologies of order type $\mathfrak{t}$.

Cardinal characteristics of chains and antichains, and other properties of a poset of group topologies were investigated in~\cite{ADM, BDFW, CR, CR1, D, D1, D2}.

First we define two frequently used weak topologies $\tau_L$ and $\tau_R$ on the semigroup $\C^0$. All non-zero elements are isolated in both of the above topologies and
\begin{itemize}
  \item the family $\mathcal{B}_L(0)=\{A_n\mid n\in\omega\}$, where $A_n=\{0\}\cup\{(k,m)\mid k>n, m\in\omega\}$, is an open neighborhood base at $0$ of the topology $\tau_L$;
  \item the family $\mathcal{B}_R(0)=\{B_n\mid n\in\omega\}$, where $B_n=\{0\}\cup\{(k,m)\mid m>n, k\in\omega\}$, is an open neighborhood base at $0$ of the topology $\tau_R$.
\end{itemize}

Observe that $\tau_{\min}=\tau_L{\vee}\tau_R$.

A semigroup topology $\tau$ on $\C^0$ is called {\em minimal} if there exists no semigroup topology on $\C^0$ which is strictly contained in $\tau$.

\begin{lemma}\label{lemma2}
$\tau_L$ and $\tau_R$ are minimal semigroup topologies on $\C^0$. 
\end{lemma}

\begin{proof}
We shall prove the minimality of $\tau_{L}$. In the case of $\tau_{R}$ the proof is similar.

It is easy to check that that  $(i,j)\cdot A_{n+j}\subseteq A_{n}$, $A_{n}\cdot (i,j)\subseteq A_{n}$ and $A_n\cdot A_n\subseteq A_n$ for all $i,j,n\in\omega$. Hence $(\mathcal{C}^0,\tau_L)$ is a topological semigroup.

Suppose to the contrary that there exists a Hausdorff semigroup topology $\tau$ on $\mathcal{C}^0$ which is strictly contained in $\tau_L$.
Then there exists $A_n\in \mathcal{B}_L(0)$ such that the set $U\setminus A_n$ is infinite for any open neighbourhood $U$ of zero in $(\mathcal{C}^0,\tau)$. The pigeonhole principle implies that there exists a non-negative integer $i_0\leq n$  such that the set $U\cap\{(i_0,j)\mid j\in \omega\}$ is infinite for each open neighborhood $U\in \tau$ of $0$. The continuity of the semigroup operation in $(\mathcal{C}^0,\tau)$ yields an open neighborhood $V\in \tau$ of $0$ such that $V\cdot V\subset \mathcal{C}^0\setminus \{(i_0,i_0)\}$. Since $V$ contains some $A_m$ and the set $V\cap\{(i_0,j)\mid j\in \omega\}$ is infinite, there exists a positive integer $j>m$ such that $(i_0,j)\in V$ and $(j,i_0)\in V$. Hence $(i_0,i_0)=(i_0,j)\cdot (j,i_0)\in V\cdot V\subset \mathcal{C}^0\setminus \{(i_0,i_0)\}$ which provides a contradiction.
\end{proof}

\begin{problem}
Do there exist other minimal semigroup topologies on $\mathcal{C}^0$?
\end{problem}

\begin{lemma}\label{lemma3}
If $\tau$ is a semigroup topology on $\mathcal{C}^0$ such that $\tau\in {\downarrow}^{\circ}\tau_{\min}$ then either $\tau=\tau_L$ or  $\tau=\tau_R$.
\end{lemma}

\begin{proof}
Let $\tau$ be a semigroup topology on $\mathcal{C}^0$ such that $\tau \in {\downarrow}^{\circ}\tau_{\min}$. Then there exists $C_n\in \mathcal{B}_{\min}$ such that the set $U\setminus C_n$ is infinite for each open neighborhood $U\in\tau$ of $0$. The pigeonhole principle implies that there exists a non-negative integer $i_0\leq n$ such that at least one of the following two cases holds:
\begin{itemize}
\item[(1)] the set $U\cap\{(i_0,j)\mid j\in \omega\}$ is infinite for each open neighborhood $U$ of $0$;
\item[(2)] the set $U\cap\{(j,i_0)\mid j\in \omega\}$ is infinite for each open neighborhood $U$ of $0$.
\end{itemize}

Consider case (1). Fix an arbitrary $n\in\omega$ and an open neighborhood $U\in\tau$ of $0$. Since $(n,i_0)\cdot 0=0$ the continuity of the semigroup operation in $(\mathcal{C}^0,\tau)$ yields an open neighborhood $V$ of $0$ such that $(n,i_0)\cdot V\subset U$. Since the set $V\cap \{(i_0,j)\mid j\in \omega\}$ is infinite and $(n,i_0)\cdot (i_0,j)=(n,j)$ for each $n,j\in\w$, the set $U\cap\{(n,j)\mid j\in \omega\}$ is infinite as well. Hence $0$ is an accumulation point of the set $\{(n,j)\mid j\in \omega\}$ for each $n\in\omega$. Using one more time the continuity of the semigroup operation in $(\mathcal{C}^0,\tau)$ we can find an open neighborhood $W\in\tau$ of $0$ such that $W\cdot W\subset U$. Since $\tau \subset \tau_{\min}$ we obtain that there exists $m\in\omega$ such that $C_m\subset W$. Recall that $0$ is an accumulation point of the set $\{(n,j)\mid j\in\omega\}$ for each $n\in \omega$. Hence for each $n\in\omega$ we can find a positive integer $j_n>m$ such that $(n,j_n)\in W$. Observe that for each $n\in \omega$ the set $\{(j_n,k)\mid k>m\}\subset W$. Then for each $n\in\omega$ the following inclusion holds:
$$\{(n,k)\mid k>m\}=(n,j_n)\cdot \{(j_n,k)\mid k>m\}\subset W\cdot W\subset U.$$
Hence for each open neighborhood $U\in\tau$ of $0$ there exists $m\in\omega$ such that $B_m\subset U$ which implies that
$U\in\tau_R$ and $\tau\subset \tau_R$. By lemma~\ref{lemma2}, $\tau=\tau_R$.

Similar arguments imply that $\tau=\tau_L$ provided that case (2) holds.
\end{proof}

Now we are going to describe the sublattices ${\downarrow}\tau_L$ and ${\downarrow}\tau_R$ of $\mathcal{W}$.
Let $\tau$ be an arbitrary shift-continuous topology on $\mathcal{C}^0$ such that $\tau\in {\downarrow}^{\circ}\tau_L$.
For each open neighborhood $U\in\tau$ of $0$ and $i\in\omega$ put $F^U_i=\{n\in\omega\mid (i,n)\in U\}$.


\begin{lemma}\label{lemma3a}
Let $\tau\in {\downarrow}^\circ \tau_L$. Then for each $i\in \omega$ the set $\mathcal{F}_i=\{F^U_i\mid 0\in U\in \tau\}$ is a filter on $\omega$. Moreover, $\mathcal{F}_i=\mathcal{F}_j$ for each $i,j\in \omega$.
\end{lemma}

\begin{proof}
Since $\tau\in {\downarrow}^{\circ}\tau_L$ there exists $n\in\omega$ such that the set $U\setminus A_n$ is infinite for each open neighborhood $U\in\tau$ of $0$.
Similar arguments as in the proof of Lemma~\ref{lemma3} imply that
$0$ is an accumulation point of the set $\{(k,n)\mid n\in\omega\}$
for each $k\in \omega$.
Fix any $i\in\omega$. Since the intersection of two neighborhoods of $0$ is a neighborhood of $0$, the family $\mathcal{F}_i$ is closed under finite intersections. Fix an arbitrary subset $A\subset \omega$ such that $F^U_i\subset A$ for some $U\in \tau$. Since each non-zero point is isolated in $(\mathcal{C}^0,\tau)$ the set $V=U\cup \{(i,n)\mid n\in A\}$ is an open neighborhood of $0$ such that $F^V_i=A$. Hence the family $\mathcal{F}_i$ is a filter for each $i\in\omega$.

Fix an arbitrary $i,j\in\omega$. Without loss of generality we can assume that $i<j$.
To prove that $\mathcal{F}_i=\mathcal{F}_j$ it is sufficient to prove the following statements:
\begin{itemize}
\item[(1)] for each $F_i^U\in \mathcal{F}_{i}$ there exists $F_j^V\in \mathcal{F}_{j}$ such that $F_j^V\subset F_i^U$;
\item[(2)] for each $F_j^U\in \mathcal{F}_{j}$ there exists $F_i^V\in \mathcal{F}_{i}$ such that $F_i^V\subset F_j^U$.
\end{itemize}

Consider statement (1). Fix an arbitrary element $F_i^U\in \mathcal{F}_{i}$. The separate continuity of the semigroup operation in $(\mathcal{C}^0,\tau)$ yields an open neighborhood $V$ of $0$ such that $(0,j-i)\cdot V\subset U$.
Observe that $(0,j-i)\cdot (j,n)=(i,n)$ for each $n\in\omega$ which implies that $F_j^V\subset F_i^U$.

Consider statement (2). Fix an arbitrary element $F_j^U\in \mathcal{F}_{j}$.
The separate continuity of the semigroup operation in $(\mathcal{C}^0,\tau)$ yields an open neighborhood $V$ of $0$ such that $(j-i,0)\cdot V\subset U$. Observe that $(j-i,0)\cdot (i,n)=(j,n)$ for each $n\in\omega$ which implies that $F_i^V\subset F_j^U$.

Hence $\mathcal{F}_i=\mathcal{F}_j$, for each $i,j\in\omega$.
\end{proof}

By Lemma~\ref{lemma3a},
each shift-continuous topology $\tau\in{\downarrow}^{\circ}\tau_L$
generates a unique filter on $\omega$ which we denote by $\mathcal{F}_{\tau}$.

For each $n\in\mathbb{Z}$ and $A\subset \omega$ the set $\{n+x\mid x\in A\}\cap\omega$ is denoted by $n+A$.
A filter $\mathcal{F}$ on the set $\omega$ is called {\em shift-invariant} if it satisfies the following conditions:
\begin{itemize}
\item each cofinite subset of $\omega$ belongs to $\mathcal{F}$;
\item for each $F\in\mathcal{F}$ and $n\in \mathbb{Z}$ there exists $H\in \mathcal{F}$ such that $n+H\subset F$.
\end{itemize}

For each filter $\mathcal{F}$ on $\omega$ and $n\in\mathbb{Z}$ the filter which is generated by the family $\{n+F\mid F\in\mathcal{F}\}$ is denoted by $n+\mathcal{F}$.

\begin{lemma}
A free filter $\mathcal{F}$ on the set $\omega$ is shift-invariant iff $\mathcal{F}=n+\mathcal{F}$ for each $n\in \mathbb{Z}$.
\end{lemma}
\begin{proof}
Let $\mathcal{F}$ be a shift-invariant filter on $\omega$ and $n\in\mathbb{Z}$. Fix any $F\in\mathcal{F}$. There exists an element $H\in\mathcal{F}$ such that $n+H\subset F$. Hence $\mathcal{F}\subset n+\mathcal{F}$. Fix any set $n+F\in n+\mathcal{F}$. Since $\mathcal{F}$ is shift-invariant there exists $H\in\mathcal{F}$ such that $-n+H\subset F$ and $k>n$ for each $k\in H$. Then $H\subset n+F$ witnessing that $n+\mathcal{F}\subset \mathcal{F}$. Hence $\mathcal{F}=n+\mathcal{F}$ for each $n\in \mathbb{Z}$.

Let $\mathcal{F}$ be a free filter on $\omega$ such that $\mathcal{F}=n+\mathcal{F}$ for each $n\in \mathbb{Z}$. Since $\mathcal{F}$ is free every cofinite subset of $\omega$ belongs to $\mathcal{F}$. Fix any $n\in\mathbb{Z}$ and $F\in \mathcal{F}$. Since $\mathcal{F}=n+\mathcal{F}$ there exists $H\in \mathcal{F}$ such that $n+H\subset F$.
\end{proof}

By $\mathcal{SIF}$ we denote the set of all shift-invariant filters on $\omega$ endowed with a following partial order:
$\mathcal{F}_1\leq \mathcal{F}_2$ iff $\mathcal{F}_1\subseteq \mathcal{F}_2$, for each $\mathcal{F}_1,\mathcal{F}_2\in \mathcal{SIF}$.

\begin{lemma}\label{lemma4}
Each $\tau\in{\downarrow}^{\circ}\tau_L$ generates a shift-invariant filter $\mathcal{F}_{\tau}$ on $\omega$. Moreover, $\mathcal{F}_{\tau_1}\neq \mathcal{F}_{\tau_2}$ for any distinct shift-continuous topologies $\tau_1$ and $\tau_2$ on $\mathcal{C}^0$ which belongs to ${\downarrow}^{\circ}\tau_L$.
\end{lemma}

\begin{proof}
Observe that each open neighborhood $U\in \tau$ of $0$ is of the form $U=\cup_{i=0}^n\{(i,n)\mid n\in F^U_i\}\cup A_n$ where $F^U_i\in \mathcal{F}_{\tau}$ for each $i\leq n$ and $n\in\omega$.
By Lemma~\ref{lemma3a},
the set $F=\cap_{i=0}^n F^U_i$
 belongs to $\mathcal{F}_{\tau}$. Then the set
$U_{F,n}=\{(i,k)\mid i\leq n\hbox{ and }k\in F\}\cup A_n$
 is an open neighborhood of $0$ which is contained in $U$.
Hence the family $\mathcal{B}(0)=\{U_{F,n}\mid F\in\mathcal{F}_{\tau}\hbox{ and }n\in\omega\}$
forms an open neighborhood base at $0$ of the topology $\tau$.

The Hausdorfness of $(\mathcal{C}^0,\tau)$ implies that each cofinite subset of $\omega$ belongs to $\mathcal{F}$.

Fix an arbitrary $n\in \mathbb{Z}$ and a basic open neighborhood $U_{F,0}$ of $0$.

If $n>0$, then the separate continuity of the semigroup operation in $(\C^0,\tau)$ yields a basic open neighborhood $U_{H,m}$ of $0$ such that $(0,n)\cdot U_{H,m}\subset U_{F,0}$. Observe that $\{(0,k)\mid k\in H\}\subset U_{H,m}$ and
$$(0,n)\cdot \{(0,k)\mid k\in H\}=\{(0,k+n)\mid k\in H\}=\{(0,k)\mid k\in H+n\}\subset U_{F,0}.$$
Hence $H+n\subset F$.

Assume that $n<0$. Since $(\C^0,\tau)$ is a Hausdorff semitopological semigroup there exists a basic open neighborhood $U_{H,m}$ of $0$ such that $H\cap\{0,\ldots,n\}=\emptyset$ and $U_{H,m}\cdot(|n|,0)\subset U_{F,0}$. Observe that $\{(0,k)\mid k\in H\}\subset U_{H,m}$ and
$$\{(0,k)\mid k\in H\}\cdot(|n|,0) =\{(0,k-|n|)\mid k\in H\}=\{(0,k)\mid k\in H+n\}\subset U_{F,0}.$$
Hence $H+n\subset F$.

Observe that the second part of the lemma follows from the description of topologies $\tau_1,\tau_2\in {\downarrow}^\circ\tau_L$ and from the definition of filters $\mathcal{F}_{\tau_1}, \mathcal{F}_{\tau_2}$.
\end{proof}

Let $\mathcal{F}$  be a shift-invariant filter on $\omega$. By $\tau_{\mathcal{F}}^L$
we denote a topology on $\mathcal{C}^0$ which is defined as follows:
each non-zero element is isolated in $(\mathcal{C}^0,\tau_{\mathcal{F}}^L)$ and the family
$\mathcal{B}=\{U_{F,n}\mid F\in \mathcal{F}\hbox{ and } n\in\omega\}$ where $U_{F,n}=\{(i,k)\mid i\leq n\hbox{ and }k\in F\}\cup A_n$
forms an open neighborhood base of $0$ in $(\mathcal{C}^0,\tau_{\mathcal{F}}^L)$.

\begin{lemma}\label{lemma5}
For each shift-invariant filter $\mathcal{F}$ on $\omega$ the topology $\tau_{\mathcal{F}}^L$ is shift-continuous and belongs to ${\downarrow}^\circ\tau_L$. Moreover, if $\mathcal{F}_{1}$ and $\mathcal{F}_{2}$ are distinct shift-invariant filters on $\omega$ then $\tau_{\mathcal{F}_1}^L\neq\tau_{\mathcal{F}_2}^L$.
\end{lemma}
\begin{proof}
Observe that the second part of the statement of the lemma follows from the definition of topologies $\tau_{\mathcal{F}_{1}}^L$ and $\tau_{\mathcal{F}_{2}}^L$.

The definition of the topology $\tau_{\mathcal{F}}^L$ implies that $\tau_{\mathcal{F}}^L\in{\downarrow}^\circ\tau_L$.

Let $\mathcal{F}$ be a shift-invariant filter on $\omega$. Since the filter $\mathcal{F}$ contains all cofinite subsets of $\omega$ the definition of the topology $\tau_{\mathcal{F}}^L$ implies that the space $(\mathcal{C}^0,\tau_{\mathcal{F}}^L)$ is Hausdorff. Recall that the bicyclic monoid $\mathcal{C}$ is generated by two elements $(0,1)$ and $(1,0)$ and $\mathcal{C}$ is the discrete subset of $(\mathcal{C}^0,\tau_{\mathcal{F}}^L)$. Hence to prove the separate continuity of the semigroup operation in $(\mathcal{C}^{0},\tau_{\mathcal{F}}^L)$ it is sufficient to check it in the following four cases:
\begin{itemize}
\item[(1)] $(0,1)\cdot 0=0$;
\item[(2)] $(1,0)\cdot 0=0$;
\item[(3)] $0\cdot (0,1)=0$;
\item[(4)] $0\cdot (1,0)=0$.
\end{itemize}

Fix an arbitrary basic open neighborhood $U_{F,n}$ of $0$.

(1) Since the filter $\mathcal{F}$ is shift-invariant there exists $H\in\mathcal{F}$ such that $H\subset F$ and $H+1\subset F$. It is easy to check that $(0,1)\cdot U_{H,n+1}\subset U_{F,n}$.

(2) It is easy to check that $(1,0)\cdot U_{F,n}\subset U_{F,n}$.

(3) Since the filter $\mathcal{F}$ is shift-invariant there exists $H\in\mathcal{F}$ such that $H\subset F$ and $H+1\subset F$. It is easy to check that $U_{H,n}\cdot (0,1)\subset U_{F,n}$.

(4) Since the filter $\mathcal{F}$ is shift-invariant there exists $H\in\mathcal{F}$ such that $H\subset F$ and $H-1\subset F$. It is easy to check that $U_{H,n}\cdot (1,0)\subset U_{F,n}$.

Hence the topology $\tau_{\mathcal{F}}^L$ is shift-continuous.
\end{proof}

Put $\mathcal{SIF}^1=\mathcal{SIF}\sqcup\{1\}$ and extend the partial order $\leq$ on $\mathcal{SIF}^1$ as follows: $\mathcal{F}\leq 1$ for each $\mathcal{F}\in \mathcal{SIF}$.


\begin{theorem}\label{th1}
The posets ${\downarrow}\tau_L$ and $\mathcal{SIF}^1$ are order isomorphic.
\end{theorem}
\begin{proof}
By Lemma~\ref{lemma4}, for each topology $\tau\in{\downarrow}^\circ\tau_L$ there exists a shift-invariant filter $\mathcal{F}\in \mathcal{SIF}$ such that $\tau=\tau_{\mathcal{F}}^L$. We define the map $f:{\downarrow}\tau_L\rightarrow \mathcal{SIF}^1$ as follows:
\begin{itemize}
\item $f(\tau_{\mathcal{F}}^L)=\mathcal{F}$ for each topology $\tau_{\mathcal{F}}^L\in{\downarrow}^\circ\tau_L$;
\item $f(\tau_L)= 1$.
\end{itemize}
Let $\mathcal{F}_1\in\mathcal{SIF}$ and $\mathcal{F}_2\in\mathcal{SIF}$.
Observe that $\mathcal{F}_1\leq \mathcal{F}_2$ iff $\tau_{\mathcal{F}_1}^L\leq \tau_{\mathcal{F}_2}^L$.
Lemmas~\ref{lemma4} and~\ref{lemma5} imply that the map $f$ is an order isomorphism between posets ${\downarrow}\tau_L$ and $\mathcal{SIF}^1$.
\end{proof}

\begin{corollary}
The poset $\mathcal{SIF}^1$ is a complete lattice.
\end{corollary}

Let $\tau$ be an arbitrary shift-continuous topology on $\mathcal{C}^0$ such that $\tau\in {\downarrow}^{\circ}\tau_R$.
For each open neighborhood $U\in\tau$ of $0$ and $i\in\omega$ put $G^U_i=\{n\in\omega\mid (n,i)\in U\}$.

The proof of the following lemma is similar to Lemma~\ref{lemma3a}.
\begin{lemma}\label{lemma3b}
Let $\tau\in {\downarrow}^\circ \tau_R$. Then for each $i\in \omega$ the set $\mathcal{G}_i=\{G^U_i\mid 0\in U\in \tau\}$ is a filter on $\omega$. Moreover, $\mathcal{G}_i=\mathcal{G}_j$ for each $i,j\in \omega$.
\end{lemma}

Let $\mathcal{G}$  be a shift-invariant filter on $\omega$. By $\tau_{\mathcal{G}}^R$
we denote a topology on $\mathcal{C}^0$ which is defined as follows:
each non-zero element is isolated in $(\mathcal{C}^0,\tau_{\mathcal{G}}^R)$ and the family
$\mathcal{B}=\{U_{G,n}\mid G\in \mathcal{G}\hbox{ and } n\in\omega\}$ where $U_{G,n}=\{(k,i)\mid i\leq n\hbox{ and }k\in G\}\cup B_n$
forms an open neighborhood base of $0$ in $(\mathcal{C}^0,\tau_{\mathcal{G}}^R)$.

Similarly it can be proved analogues of lemmas~\ref{lemma4} and~\ref{lemma5} for a topology $\tau\in{\downarrow}^\circ\tau_R$.
Hence we obtain the following:

\begin{theorem}\label{th2}
The posets ${\downarrow}\tau_R$ and $\mathcal{SIF}^1$ are order isomorphic.
\end{theorem}

\begin{proposition}\label{pr2}
For each topology $\tau\in \mathcal{W}$ there exists a unique pair of topologies $(\tau_{1},\tau_{2})\in {\downarrow}\tau_L{\times}{\downarrow}\tau_R$ such that $\tau=\tau_{1}{\vee}\tau_{2}$.
\end{proposition}

\begin{proof}
Let $\tau\in {\downarrow}^\circ \tau_{\min}$. Then there exists $n\in\omega$ such that for each open neighborhood $U\in\tau$ of $0$ the set $U\setminus C_n$ is infinite.

It is easy to see that $\tau$ satisfies one of the following three conditions.
\begin{itemize}
\item[(1)] For every $i\leq n$ the set $\{(i,m)\mid m\in\omega\}\cap U$ is infinite for each open neighborhood $U\in\tau$ of $0$ and there exists an open neighborhood $U_0\in\tau$ of $0$ such that $\{(m,j)\mid m\in\omega\}\cap U_0=\emptyset$ for any $j\leq n$.
\item[(2)] For every $j\leq n$ the set $\{(m,j)\mid m\in\omega\}\cap U$ is infinite for each open neighborhood $U\in \tau$ of $0$ and there exists an open neighborhood $U_0\in\tau$ of $0$ such that $\{(i,m)\mid m\in\omega\}\cap V=\emptyset$ for any $i\leq n$.
\item[(3)] For every $i,j\leq n$ the sets $\{(i,m)\mid m\in\omega\}\cap U$ and $\{(m,j)\mid m\in\omega\}\cap U$ are infinite for each open neighborhood $U\in\tau$ of $0$.
\end{itemize}
Assume that $\tau$ satisfies condition (1). Fix a non-negative integer $k$ and an arbitrary open neighborhood $U\in\tau$ of $0$. Similar arguments as in the proof of lemmas~\ref{lemma3a},~\ref{lemma4} imply that there exists a shift-invariant filter $\mathcal{F}$ on $\omega$ such that the family $\mathcal{B}(0)=\{C_{F,n}\mid F\in\mathcal{F},n\in\omega\}$, where
$C_{F,n}=C_n\cup\{(i,k)\mid i\leq n\hbox{ and }k\in F\}$, is an open neighborhood base at $0$ of the topology $\tau$.

It is easy to check that $\tau= \tau_{\mathcal{F}}^L{\vee}\tau_R$. We denote such a topology $\tau$ by $\tau_{\mathcal{F},1}$.

Assume that $\tau$ satisfies condition (2). Similar arguments imply that there exists a shift stable filter $\mathcal{G}$ on $\omega$ such that the family $\mathcal{B}(0)=\{C_{G,n}\mid G\in\mathcal{G},n\in\omega\}$, where
$C_{G,n}=C_n\cup\{(k,i)\mid i\leq n\hbox{ and }k\in G\}$, is an open neighborhood base at $0$ of the topology $\tau$.

Then $\tau= \tau_L{\vee}\tau_{\mathcal{G}}^R$. We denote such a topology $\tau$ by $\tau_{1,\mathcal{G}}$.

If the topology $\tau$ satisfies condition (3), then there exist shift-invariant filters $\mathcal{F},\mathcal{G}$ on $\omega$ such that $\tau=\tau_{\mathcal{F}}^L{\vee}\tau_{\mathcal{G}}^R$. We denote such a topology $\tau$ by $\tau_{\mathcal{F},\mathcal{G}}$.

Recall that $\tau_{\min}=\tau_L{\vee}\tau_R$. For convenience we denote $\tau_{\min}$ by $\tau_{1,1}$.
\end{proof}


\begin{theorem}\label{th3}
The poset $\mathcal{W}$ is order isomorphic to the poset $\mathcal{SIF}^1{\times}\mathcal{SIF}^1$.
\end{theorem}

\begin{proof}
By Proposition~\ref{pr2}, each topology $\tau\in\mathcal{W}$ is of the form $\tau_{x,y}$, where $x,y\in \mathcal{SIF}^1$. The routine verifications show that the map $f:\mathcal{W}\rightarrow \mathcal{SIF}^1{\times}\mathcal{SIF}^1$, $f(\tau_{x,y})=(x,y)$ is an order isomorphism.
\end{proof}

An inverse semigroup $S$ is called {\em quasitopological} if it is semitopological and the inversion is continuous in $S$. A topology $\tau$ on an inverse semigroup $S$ is called {\em quasisemigroup} if $(S,\tau)$ is a quasitopological semigroup.

By $\mathcal{W}_q$ we denote the set of all weak quasisemigroup topologies on $\C^0$. Obviously, $\mathcal{W}_q$ is a sublattice of $\mathcal{W}$. Fix a weak topology $\tau=\tau_{x,y}$ where $x,y\in \mathcal{SIF}^1$ (see the proof of Proposition~\ref{pr2}). Observe that $(n,m)^{-1}=(m,n)$ for each element $(n,m)\in \C$. At this point it is easy to see that the topology $\tau_{x,y}$ is quasisemigroup iff $x=y$. Hence we obtain the following:

\begin{proposition}
The lattice $\mathcal{W}_q$ is isomorphic to the lattice $\mathcal{SIF}^1$.
\end{proposition}


At the end of this section we prove a nice complete-like property of weak topologies on $\C^0$.

A semitopological semigroup $X$ is called {\em absolutely H-closed} if for any continuous homomorphism $h$ from $X$ into a Hausdorff semitopological semigroup $Y$ the image $h(X)$ is closed in $Y$.

The next proposition complements results about complete polycyclic monoids $\mathcal{P}_k$, $k>1$ obtained in~\cite{Bardyla-Gutik-2016}.

\begin{proposition}\label{hcl}
Let $(\C^0,\tau)$ be a semitopological semigroup such that $0$ is an accumulation point of the set $E(\C)$. Then $(\C^0,\tau)$ is absolutely H-closed.
\end{proposition}
\begin{proof}
By $\sigma$ we denote the least group congruence on $\mathcal{C}$.  According to Theorem~3.4.5 from~\cite{Lawson-2009}, $(a,b)\sigma (c,d)$ iff $a-b=c-d$ for any $a,b,c,d\in\w$, every congruence on the bicyclic monoid is a group congruence and  $\mathcal{C}/\sigma$ is isomorphic to the additive group of integers. For each $k\in\mathbb{Z}$ put $[k]=\{(a,b)\in \mathcal{C}\mid a-b=k\}$.

Fix any $k\in\mathbb{Z}$ and an open neighborhood $U$ of $0$.
If $k>0$, then the separate continuity of the semigroup operation in $(C^0,\tau)$ yields an open neighborhood $V$ of $0$ such that $(k,0)\cdot V\subset U$. Since $0$ is an accumulation point of the set $E(\C)$ there exists an infinite subset $A\subset \w$ such that
$\{(n,n)\mid n\in A\}\subset V$. Then $(k,0)\cdot \{(n,n)\mid n\in A\}=\{(k+n,n)\mid n\in A\}\subset U$. Hence $0$ is an accumulation point of the set $[k]$ for each $k\geq 0$.

If $k<0$ then the separate continuity of the semigroup operation in $(C^0,\tau)$ yields an open neighborhood $V$ of $0$ such that $(0,k)\cdot V\subset U$. Similarly it can be shown that $0$ is an accumulation point of the set $[k]$ for each $k<0$.

Assume that $h$ is a continuous homomorphism from $\C^0$ into a Hausdorff semitopological semigroup $X$. If there exists $(n,m)\in\C$ such that $h(n,m)=h(0)$ then
$$h((0,0))=h((0,n)\cdot (n,m)\cdot (m,0))=h(0,n)\cdot h(0)\cdot h(m,0)=h((0,n)\cdot 0\cdot(m,0))=h(0).$$
In this case the map $h$ is annihilating.

Otherwise, by Theorem~3.4.5 from~\cite{Lawson-2009}, there are three cases to consider:
\begin{itemize}
\item[(1)] the image $h(\C^0)$ is finite;
\item[(2)] the image $h(\C^0)$ is isomorphic to the additive group of integers with an adjoint zero;
\item[(3)] $h$ is injective, i.e., $h(\C^0)$ is isomorphic to $\C^0$.
\end{itemize}

(1) The Hausdorffness of $X$ implies that $h(\C^0)$ is a closed subset of $X$.

(2) Observe that $0$ is an accumulation point of each equivalence class $[k]$ of the least group congruence $\sigma$. Hence each open neighborhood $U$ of $h(0)$ in $X$ contains the set $h(\C)$ which contradicts to the Hausdorffness of $X$. Hence this case is not possible.

(3) To obtain a contradiction, assume that $h(\C^0)$ is not closed in $X$. Hence there exists an element
$x\in \overline{h(\C^0)}\setminus h(\C^0)$. Each open neighborhood of $x$ contains infinitely many elements from $\C$. Hence for each open neighborhood $V$ of $x$ at least one of the following two subcases holds:
 \begin{itemize}
\item[(3.1)] the set $L_V=\{n\mid$ there exists $m$ such that $(n,m)\in V\}$ is infinite;
\item[(3.2)] the set $R_V=\{m\mid$ there exists $n$ such that $(n,m)\in V\}$ is infinite.
\end{itemize}

(3.1) We claim that $(k,k)\cdot x=x$ for each $k\in\w$. Indeed, fix any $k\in\w$ and observe that $(k,k)\cdot (n,m)=(n,m)$ for each $n\geq k$. Hence the set $((k,k)\cdot V)\cap V$ is infinite for each open neighborhood $V$ of $x$ witnessing that $(k,k)\cdot x=x$.

Fix any open neighborhood $U$ of $0$ which does not contain $x$. The separate continuity of the semigroup operation in $X$ implies that $0\cdot x=x\cdot 0=0$. Hence there exists an open neighborhood $W\subset U$ of $0$ such that $W\cdot x\subset U$. Fix any idempotent $(k,k)\in W$
(it is possible since $0$ is an accumulation point of the set $E(\C)=\{(k,k)\mid k\in\w\}$). The above claim implies that $x=(k,k)\cdot x\in W\cdot x\subset U$ which contradicts to the choice of the set $U$.

(3.2) Analogous it can be showed that $x\cdot (k,k)=x$ for each $k\in\w$. At this point the contradiction can be obtained similarly as in (3.1).
\end{proof}

Proposition~\ref{hcl} provides the following:
\begin{corollary}
For each weak topology $\tau$ the semitopological semigroup $(\C^0,\tau)$ is absolutely H-closed.
\end{corollary}

\section{Cardinal characteristics of the lattice $\mathcal{W}$}
By $[\w]^\w$ we denote the family of all infinite subsets of $\w$. We write $A\subset^* B$ if $|A\setminus B|<\w$.
Let $\mathcal{F}$ be a filter on $\omega$. The cardinal $\chi(\mathcal{F})=\min\{|\mathcal{B}|: \mathcal{B}$ is a base of the filter $\mathcal{F}\}$ is called the {\em character} of the filter $\mathcal{F}$.
A filter $\mathcal{F}$ is called {\em first-countable} if $\chi(\mathcal{F})=\omega$.

For each $a,b\in \omega$ by $[a,b]$ we denote the set $\{n\in\omega\mid a\leq n\leq b\}$. Observe that $[a,b]=\emptyset$ if $a>b$. Let $A$ be an infinite subset of $\omega$. For each $k\in\omega$ put $F_{A,k}=\cup_{n\in A}[n!-n+k,n!+n-k]$. Since $F_{A,k}\cap F_{A,n}=F_{A,\max\{k,n\}}$ the family $\{F_{A,k}\mid k\in\omega\}$ is closed under finite intersections. By $\mathcal{F}_A$ we denote the filter on $\omega$ which base consists of the sets $F_{A,k}$, $k\in\omega$.

\begin{lemma}\label{lc1} Let $A,B\in [\omega]^{\w}$. Then the following statements hold:
 \begin{itemize}
 \item [(1)] the filter $\mathcal{F}_A$ is shift-invariant;
 \item [(2)] if $A$ and $B$ are almost disjoint, then the filters $\mathcal{F}_{A}$ and $\mathcal{F}_{B}$ are incomparable in the lattice $\mathcal{SIF}$.
 \item [(3)] if $A\subset^* B$ then $\mathcal{F}_B\leq\mathcal{F}_A$;
 \item [(4)] if $|A\setminus B|=\omega$ then $\mathcal{F}_A\neq\mathcal{F}_B$
 \item [(5)] $(\mathcal{C}^0,\tau_{\mathcal{F}_A}^L)$ and $(\mathcal{C}^0,\tau_{\mathcal{F}_A}^R)$ are first-countable topological spaces.
 \end{itemize}
\end{lemma}
\begin{proof}
(1). Fix any element $F_{A,n}\in\mathcal{F}_{A}$, $m\in \mathbb{Z}$ and $k\in\omega$. It is easy to check that $m+F_{A,n+|m|}\subset F_{A,n}$ and $F_{A,k+1}\subset \omega\setminus [0,k]$. Hence the filter $\mathcal{F}_{A}$ is shift-invariant.

(2). There exists $n\in\omega$ such that $A\cap B\subset [0,n]$. It is easy to check that $F_{A,n+1}\cap F_{B,n+1}=\emptyset$ which implies that the filters $\mathcal{F}_A$ and $\mathcal{F}_B$ are incomparable in the poset $\mathcal{SIF}$.

(3). There exists $n\in\omega$ such that $A\setminus B\subset [0,n]$. It is easy to check that $F_{A,k}\subset F_{B,k}$ for each $k\geq n+1$ which implies that $\mathcal{F}_B\leq\mathcal{F}_A$.

Statement (4) follows from the definition of the filters $\mathcal{F}_A$ and $\mathcal{F}_B$.

Observe that the filter $\mathcal{F}_A$ is first-countable. At this point statement (5) follows from the definition of the topologies $\tau_{\mathcal{F}_A}^L$ and $\tau_{\mathcal{F}_A}^R$.
\end{proof}

By $\mathcal{SCT}_{\omega}$ ($\mathcal{W}_{\omega}$, resp.) we denote the set of all (weak, resp.) Hausdorff shift-continuous first-countable topologies on $\mathcal{C}^0$. It is easy to check that $\mathcal{SCT}_{\omega}$ is a sublattice of $\mathcal{SCT}$.
A subset $A$ of a poset $X$ is called an {\em antichain} if each two distinct elements of $A$ are incomparable in $X$.

A set $A$ is called a {\em pseudo-intersection} of a family $\mathcal{F}\subset [\w]^{\w}$ if $A\subset^*F$ for each $F\in \mathcal{F}$.
A {\em tower} is a set $\mathcal{T}\subset [\w]^{\w}$ which is well-ordered with respect to the relation
defined by $x\leq y$ iff $y\subset^* x$. It is called {\em maximal} if it cannot be further extended, i.e. it has no
pseudointersection.

Denote $\mathfrak{t}=\min\{|\mathcal{T}|:\mathcal{T}$ is a maximal tower$\}$. By~\cite[Theorem 3.1]{vd}, $\omega_1\leq \mathfrak{t}\leq \mathfrak{c}$. Put $\hat{\mathfrak{t}}=\sup\{|\mathcal{T}|:\mathcal{T}$ is a maximal tower$\}$. Obviously, $\mathfrak{t}\leq \hat{\mathfrak{t}}\leq \mathfrak{c}$.

\begin{theorem}\label{card0} The poset $\mathcal{W}_{\omega}$ has the following properties:
\begin{itemize}
\item [(1)] $\mathcal{W}_{\omega}$ contains an antichain of cardinality $\mathfrak{c}$;
\item [(2)] For each ordinal $\kappa\in  \hat{\mathfrak{t}}$ the poset $\mathcal{W}_{\omega}$ contains a well-ordered chain of order type $\kappa$;
\item [(3)] $|\mathcal{W}_{\omega}|=|\mathcal{SCT}_{\omega}|=\mathfrak{c}$.
\end{itemize}
\end{theorem}

\begin{proof}
(1) Fix any almost disjoint family $\mathcal{A}\subset [w]^\w$ such that $|\mathcal{A}|=\mathfrak{c}$. By~\cite[Theorem~1.3]{Kunen}, such a family exists. Statement (2) of Lemma~\ref{lc1} implies that the set $\{\mathcal{F}_A\mid A\in\mathcal{A}\}$ forms an antichain in the poset $\mathcal{SIF}$. By Theorem~\ref{th1} (resp., \ref{th2}), the set $\{\tau_{\mathcal{F}_{A}}^L\mid A\in\mathcal{A}\}$ (resp., $\{\tau_{\mathcal{F}_{A}}^R\mid A\in\mathcal{A}\}$) is an antichain in the lattice $\mathcal{W}$. By statement (5) of Lemma~\ref{lc1}, the sets $\{\tau_{\mathcal{F}_{A}}^L\mid A\in\mathcal{A}\}$ and $\{\tau_{\mathcal{F}_{A}}^R\mid A\in\mathcal{A}\}$ are contained in $\mathcal{W}_\omega$

(2) Fix any ordinal $\kappa\in  \hat{\mathfrak{t}}$. By the definition of $\hat{\mathfrak{t}}$ there exists a tower $\mathcal{T}=\{T_{\alpha}\}_{\alpha\in \lambda}$ of length $\lambda>\kappa$. Observe that $|T_{\alpha}\setminus T_{\beta}|=\omega$ for each $\alpha< \beta< \lambda$. Statements (3), (4) and (5) of Lemma~\ref{lc1} implies that the sets $\{\tau_{\mathcal{F}_{T_{\alpha}}}^L\mid \alpha\in\kappa\}$ and $\{\tau_{\mathcal{F}_{T_{\alpha}}}^R\mid \alpha\in\kappa\}$ are well-ordered chains in $\mathcal{W}_\omega$ of order type $\kappa$.

(3) Observe that the cardinality of the set of all first countable filters on $\omega$ is equal to $\mathfrak{c}$. Statement (1) implies that $|\mathcal{SCT}_{\omega}|=|\mathcal{W}_{\omega}|=\mathfrak{c}$.
\end{proof}

Now we are going to show that each free filter on $\omega$ generates a shift-invariant filter on $\omega$.
Fix an arbitrary free filter $\mathcal{G}$ on $\omega$. For each $G\in\mathcal{G}$ and $k\in\omega$ put $F_{G,k}=\cup_{n\in G}[n!-n+k,n!+n-k]$. Observe that $F_{G,k}\cap F_{H,n}=F_{H\cap G,\max\{k,n\}}$ for each $G,H\in \mathcal{G}$ and $k,n\in \omega$. Since $H\cap G\in\mathcal{G}$ we obtain that the family $\{F_{G,k}\mid G\in\mathcal{G},k\in\omega\}$ is closed under finite intersections. By $\mathcal{F}_{\mathcal{G}}$ we denote the filter on $\omega$ which base consists of the sets $F_{G,k}$, $G\in\mathcal{G}$ and $k\in\omega$. A filter $\mathcal{F}$ on $\omega$ is called an {\em ultrafilter} if for each subset $A\subset\omega$ either $A\in \mathcal{F}$ or there exists $F\in \mathcal{F}$ such that $A\cap F=\emptyset$.

\begin{lemma}\label{lc2}
Let $\mathcal{G},\mathcal{H}$ be free filters on $\omega$. Then the following statements hold:
 \begin{itemize}
 \item [(1)] the filter $\mathcal{F}_{\mathcal{G}}$ is shift-invariant;
 \item [(2)] if $\mathcal{G}$ and $\mathcal{H}$ are ultrafilters, then the filters $\mathcal{F}_{\mathcal{G}}$ and $\mathcal{F}_{\mathcal{H}}$ are incomparable in the lattice $\mathcal{SIF}$.
 \item [(3)] if $\mathcal{G}\subset \mathcal{H}$ then $\mathcal{F}_{\mathcal{G}}\subset \mathcal{F}_{\mathcal{H}}$;
 \item [(4)] if  $\mathcal{G}\neq\mathcal{H}$ then $\mathcal{F}_{\mathcal{G}}\neq \mathcal{F}_{\mathcal{H}}$;
 \item [(5)] The character of the spaces $(\mathcal{C}^0,\tau_{\mathcal{F}_{\mathcal{G}}}^L)$ and $(\mathcal{C}^0,\tau_{\mathcal{F}_{\mathcal{G}}}^R)$ is equal to the character of the filter $\mathcal{F}_{\mathcal{G}}$.
 \end{itemize}
\end{lemma}

\begin{proof}
(1) Fix any element $F_{G,n}\in\mathcal{F}_{\mathcal{G}}$, $m\in \mathbb{Z}$ and $k\in\omega$. It is easy to check that $m+F_{G,n+|m|}\subset F_{G,n}$ and $F_{G,(k+1)!}\subset \omega\setminus [0,k]$. Hence the filter $\mathcal{F}_{\mathcal{G}}$ is shift-invariant.

(2) Assume that $\mathcal{G}$ and $\mathcal{H}$ are ultrafilters on $\w$. Then there exists $A\in \mathcal{G}$ and $B\in \mathcal{H}$ such that $A\cap B=\emptyset$. It is easy to see that the sets $F_{A,2}\in\mathcal{F}_{\mathcal{G}}$ and $F_{B,2}\in\mathcal{F}_{\mathcal{H}}$ are disjoint which implies that the filters $\mathcal{F}_{\mathcal{G}}$ and $\mathcal{F}_{\mathcal{H}}$ are incomparable in the poset $\mathcal{SIF}$.

Statement (3) follows from the definition of the filters $\mathcal{F}_{\mathcal{G}}$ and $\mathcal{F}_{\mathcal{H}}$.

The proof of statement (4) is straightforward.

Statement (5) follows from the definition of the topologies $\tau_{\mathcal{F}_{\mathcal{G}}}^L$ and $\tau_{\mathcal{F}_{\mathcal{G}}}^R$.
\end{proof}

By $\mathcal{FF}$ we denote the set of all free filters on $\omega$ endowed with the natural partial order: $\mathcal{F}\leq \mathcal{G}$ iff $\mathcal{F}\subset\mathcal{G}$.

\begin{lemma}\label{lf}
The poset $\mathcal{FF}$ contains a well-ordered chain of cardinality $\mathfrak{c}$.
\end{lemma}

\begin{proof}
By Theorem 4.4.2 from~\cite{vm}, there exists a free ultrafilter $\mathcal{F}$ on $\w$ such that $\chi(\mathcal{F})=\mathfrak{c}$. Fix any base $\mathcal{B}=\{F_{\alpha}\}_{\alpha\in \mathfrak{c}}$ of $\mathcal{F}$.

Now we are going to construct a well-ordered chain $\mathcal{L}=\{\mathcal{G}_{\alpha}\}_{\alpha\in \mathfrak{c}}$ in $\mathcal{FF}$. Let $\mathcal{G}_0$ be a filter which consists of cofinite subsets of $\omega$. Since the filter $\mathcal{F}$ is free, $\mathcal{G}_0\subset \mathcal{F}$. Assume that for all ordinals $\beta<\alpha<\mathfrak{c}$ the filters $\mathcal{G}_{\beta}$ are already constructed.
There are two cases to consider:
\begin{itemize}
\item[(1)] $\alpha$ is a limit ordinal;
\item[(2)] $\alpha$ is a successor ordinal.
\end{itemize}
In case (1) put $\mathcal{G}_{\alpha}=\cup_{\beta\in \alpha}\mathcal{G}_{\beta}$.

In case (2) $\alpha=\delta+1$ for some ordinal $\delta$. Let $\gamma_{\delta}$ be the smallest ordinal such that $F_{\gamma}\notin \mathcal{G}_{\delta}$.
Let $G_{\alpha}$ be a filter generated by
the smallest family $X$ of subsets of $\omega$ such that $\mathcal{G}_{\delta}\subset X$, $F_{\gamma_{\delta}}\in X$ and $X$ is closed under finite intersections.

Observe that $|\mathcal{G}_{\delta}|\leq\max\{\omega,|\delta|\}<\mathfrak{c}$. Hence for each $\delta\in\mathfrak{c}$ the ordinal $\gamma_{\delta}$ exists, because in the other case the family $\mathcal{G}_{\delta}$ will be a base of the filter $\mathcal{F}$ of cardinality less then $\mathfrak{c}$ which contradicts our assumption.

It is easy to see that the family $\mathcal{L}=\{\mathcal{G}_{\alpha}\}_{\alpha\in\mathfrak{c}}$ is a well-ordered chain in $\mathcal{FF}$.
\end{proof}


\begin{theorem}\label{card1} The poset $\mathcal{W}$ has the following properties:
\begin{itemize}
\item [(1)] $\mathcal{W}$ contains an antichain of cardinality $2^\mathfrak{c}$;
\item [(2)] $\mathcal{W}$ contains a well-ordered chain of cardinality $\mathfrak{c}$;
\item [(3)] $|\mathcal{W}|=|\mathcal{SCT}|=2^\mathfrak{c}$.
\end{itemize}
\end{theorem}

\begin{proof}
It is well-known that there are $2^{\mathfrak{c}}$ ultrafilters on $\omega$. Hence statement (1) follows from statements (1) and (2) of Lemma~\ref{lc2} and Theorem~\ref{th1}.

Consider statement (2). By Lemma~\ref{lf}, there exists a well-ordered chain $\mathcal{L}\subset \mathcal{FF}$ of cardinality $\mathfrak{c}$. Statements (3) and (4) of Lemma~\ref{lc2} imply that the family $\mathcal{Z}=\{\mathcal{F}_{\mathcal{G}}\mid \mathcal{G}\in\mathcal{L}\}$ is a well-ordered chain in $\mathcal{SIF}$. By Theorem~\ref{th1}, the poset $\mathcal{W}$ contains a well-ordered chain of cardinality $\mathfrak{c}$.


Consider statement (3). Obviously, $|\mathcal{W}|\leq |\mathcal{SCT}|\leq 2^{\mathfrak{c}}$. Then statement (1) yields the desired equality.
\end{proof}

\end{document}